\documentclass[12pt,twoside]{amsart}
\usepackage{amsmath, amsthm, amscd, amsfonts, amssymb, graphicx}
\usepackage{enumerate}
\usepackage[colorlinks=true,
linkcolor=blue,
urlcolor=black,
citecolor=red]{hyperref}
\usepackage{mathrsfs}
\newtheorem{theorem}{Theorem}[section]
\newtheorem{lemma}{Lemma}[section]
\newtheorem{remark}{Remark}[section]
\newtheorem{definition}{Definition}[section]
\newtheorem{corollary}{Corollary}[section]

\newtheorem{proposition}{Proposition}[section]
\numberwithin{equation}{section}

\begin{document}

\title{Generalized Euclidean Operator Radius}
\author{Mohammad W. Alomari, Mohammad Sababheh, Cristian Conde, and Hamid Reza Moradi}
\subjclass[2010]{Primary 47A12, Secondary 47A30, 26A51}
\keywords{Generalized Euclidean operator, numerical radius, Euclidean operator radius, convex function}
\begin{abstract}
In this paper, we introduce the $f-$operator radius of Hilbert space operators as a generalization of the Euclidean operator radius and the $q-$operator radius. Properties of the newly defined radius are discussed, emphasizing how it extends some known results in the literature.

\end{abstract}

\maketitle
\pagestyle{myheadings}
\markboth{\centerline {Generalized Euclidean Operator Radius}}
{\centerline {M. W. Alomari et al. }}
\bigskip
\bigskip
\section{Introduction}
Let $\mathbb{B}\left( \mathscr{H} \right)$ denote the $C^*-$algebra of all bounded linear operators on  a complex Hilbert space $\mathscr{H}.$ For $T\in \mathbb{B}\left( \mathscr{H} \right),$ the operator norm and the numerical radius of $T$ are defined, respectively, by 
$$\|T\|=\sup_{\|x\|=1}\|Tx\|\;{\text{ and }}\;\omega(T)=\sup_{\|x\|=1}
|\left<Tx,x\right>|.$$

It is well known that $\omega(\cdot)$ defines a norm on $\mathbb{B}\left( \mathscr{H} \right)$, that is equivalent to the operator norm via the  relation

\begin{align}\label{eq_eq_w_norm}
\frac{1}{2}\|T\|\le \omega(T)\le \|T\|,\;T\in \mathbb{B}\left( \mathscr{H} \right).
\end{align}

It is interesting to find possible bounds of $\omega(\cdot)$ in terms of $\|\cdot\|$ since the calculations of $\|\cdot\|$ are much easier than those of $\omega(\cdot).$ We refer the reader to \cite{O-F, bou, Mor_lama, Mor_filomat, Sab_lama, sab_num_lama, sahoo, Sattari, satt_mosl_shebr, sheyb, Sab_aofa, Z-M} as a recent list of references treating numerical radius and operator norm inequalities.

Among the most well-established interesting results in this direction are the following inequalities due to Kittaneh \cite{Ki, Kil}
\begin{equation*}
\omega(T)\le \frac{1}{2}\|\;|T|+|T^*|\;\|,
\end{equation*}
\begin{equation}\label{ineq_kitt_2}
\omega^2(T)\le \frac{1}{2}\|\;|T|^2+|T^*|^2\|,
\end{equation}
and
\begin{equation}\label{ineq_kitt_3}
\omega(T)\le \frac{1}{2}\left(\|T\|+\|T^2\|^{\frac{1}{2}}\right),
\end{equation}
where $T^*$ is the adjoint operator of $T$ and $|T|=(T^*T)^{1/2}$.

Extending the numerical radius, the Euclidean operator radius of the operators $T_1,\cdots, T_n\in\mathbb{B}\left( \mathscr{H} \right)$ was defined in \cite{2} as
\[{{\omega }_{e}}\left( {{T}_{1}},\ldots ,{{T}_{n}} \right)=\underset{\left\| x \right\|=1}{\mathop{\sup }}\,{{\left( \sum\limits_{j=1}^{n}{{{\left| \left\langle {{T}_{j}}x,x \right\rangle \right|}^{2}}} \right)}^{\frac{1}{2}}}.\]
This was also generalized in \cite{3} to
\[{{\omega }_{q}}\left( {{T}_{1}},\ldots ,{{T}_{n}} \right)=\underset{\left\| x \right\|=1}{\mathop{\sup }}\,{{\left(\sum_{j=1}^{n}{{{\left| \left\langle {T_j}x,x \right\rangle \right|}^{q}}} \right)}^{\frac{1}{q}}};\text{ }q\ge 1.\]

We refer the reader to \cite{Alomari,baj, Drag_eucl,Drag_Hind, satt_mosl_shebr,sheik_mosl} as a list of references treating properties and significance of $\omega_e$ and $\omega_q.$

In the literature, it is interesting to introduce and define new related numerical radii or operator radii in a way that extends some well-known concepts. For this particular concern, we refer the reader to \cite{O-F, kais, Sab_aofa}, where a discussion of other types of numerical radii has been presented.

This paper introduces a generalized form of $\omega_e$ and $\omega_q$ that depends on a certain function $f$. It turns out that both $\omega_e$ and $\omega_q$ are special cases of this new concept, which we define as follows.

\begin{definition}\label{2}
Let $T_1,\ldots,T_n\in \mathbb{B}(\mathscr{H})$ and let $f:[0,\infty)\to [0,\infty)$ be a continuous increasing function with $f(0)=0$. We define the $f-$operator radius of the operators $T_1,\ldots,T_n$ by
\[{{\omega }_{f}}\left( {{T}_{1}},\ldots ,{{T}_{n}} \right)=\underset{\left\| x \right\|=1}{\mathop{\sup }}\,{{f}^{-1}}\left(\sum_{j=1}^{n}{f\left( \left| \left\langle {T_j}x,x \right\rangle \right| \right)} \right).\]
\end{definition}
Thus, when $f(t)=t^2$, $\omega_f=\omega_e$, and when $f(t)=t^q, \omega_f=\omega_q,$ for $q\ge 1.$

The quantities $\omega_e,\omega_q$ were defined in \cite{3, 2} as norms on $\mathbb{B}(\mathscr{H})\times\cdots\times \mathbb{B}(\mathscr{H}).$ In what follows, we show norm properties of  $\omega_f$.

It is implicitly understood that $f:[0,\infty)\to [0,\infty)$ is a continuous increasing function with $f(0)=0$, whenever we write $\omega_f.$

The Davis-Wielandt radius of $T\in \mathbb B\left( \mathscr H \right)$ is defined as
	\[d\omega \left( T \right)=\underset{\left\| x \right\|=1}{\mathop{\sup }}\,\left\{ \sqrt{{{\left| \left\langle Tx,x \right\rangle  \right|}^{2}}+{{\left\| Tx \right\|}^{4}}} \right\}.\]
It is not hard to see that $d\omega \left( T \right)$ is unitarily invariant, but it does not define a norm on $\mathbb B\left( \mathscr H \right)$.
It is well-known that
\[\max \left\{ \omega \left( T \right),{{\left\| T \right\|}^{2}} \right\}\le d\omega \left( T \right)\le \sqrt{{{\omega }^{2}}\left( T \right)+{{\left\| T \right\|}^{4}}}.\]
Putting $n=2$, ${{T}_{1}}=T$, and ${{T}_{2}}={{T}^{*}}T$, in Definition \ref{2}, we deliver
	\[\begin{aligned}
   {{\omega }_{f}}\left( T,{{T}^{*}}T \right)&=\underset{\left\| x \right\|=1}{\mathop{\sup }}\,{{f}^{-1}}\left( f\left( \left| \left\langle Tx,x \right\rangle  \right| \right)+f\left( \left| \left\langle {{T}^{*}}Tx,x \right\rangle  \right| \right) \right) \\ 
 & =\underset{\left\| x \right\|=1}{\mathop{\sup }}\,{{f}^{-1}}\left( f\left( \left| \left\langle Tx,x \right\rangle  \right| \right)+f\left( \left| \left\langle Tx,Tx \right\rangle  \right| \right) \right) \\ 
 & =\underset{\left\| x \right\|=1}{\mathop{\sup }}\,{{f}^{-1}}\left( f\left( \left| \left\langle Tx,x \right\rangle  \right| \right)+f\left( {{\left\| Tx \right\|}^{2}} \right) \right)  
\end{aligned}\]
which provides an extension of the Davis-Wielandt radius of $T$. Notice that when $f(t)=t^2$, $\omega_f(T,T^*T)=d\omega(T).$

We need the following lemmas throughout the subsequent sections. The first lemma has been a helpful tool in studying operator inequalities in the literature.
\begin{lemma}\label{lem_conv_inner}
\cite[(4.24)]{bhatia} Let $f$ be a convex function defined on a real interval $I$ and let $T\in \mathbb{B}\left(
\mathscr{H}\right)$ be a self-adjoint operator with spectrum in $I$. Then $f\left( {\left\langle {Tx,x} \right\rangle } \right) \le \left\langle {f\left( T \right)x,x} \right\rangle $ for all unit vectors $ x\in \mathscr{H}$.
\end{lemma}

The second lemma is a useful characterization of numerical radii.
\begin{lemma}\label{lemma_theta}
 \cite {Z-M} Let $T\in \mathbb{B}\left( \mathscr{H} \right)$. Then
	$$\omega(T)=\sup_{\theta \in \mathbb R}\left\|\mathfrak{R}\left(e^{\mathrm i\theta}T\right)\right\|,$$ where $\mathfrak{R}(T)$ is the real part of the operator $T$, defined by $\mathfrak{R} T=\frac{T+T^*}{2}.$
\end{lemma}

We also need the following lemma, which holds for convex functions with $f(0)\le 0.$
\begin{lemma}\label{superadditive}
If $f:[0,\infty)\to [0,\infty)$ is a convex function with $f(0)=0$, then $f$ is superadditive. That is
$$ f(a+b)\ge f(a)+f(b),$$ for $a,b\ge 0$. The inequality is reversed when $f:[0,\infty)\to [0,\infty)$ is concave, without having $f(0)=0$.
\end{lemma}

Recall that the Aluthge transform $\widetilde{T}$ of $T\in \mathbb{B}\left( \mathscr{H} \right)$ is defined by $\widetilde{T}=|T|^{\frac{1}{2}}U|T|^{\frac{1}{2}},$ where $U$ is  the partial isometry appearing in the polar decomposition $T=U|T|$ of $T$, \cite{alu}.
Yamazaki showed the following better estimates of \eqref{ineq_kitt_3} than \cite{Z-M}
\begin{equation}\label{ineq_yama_1}
\omega \left( T \right)\le \frac{1}{2}\left( \left\| T \right\|+\omega \left( \widetilde{T} \right) \right).
\end{equation}

\section{Further discussion of $\omega_f$}
In this section, we discuss the quantity $\omega_f$. This includes basic properties and possible relations with the numerical radius $\omega$ and the operator norm $\|\cdot\|$. More applications to numerical radius bounds will be discussed too.

We begin with the following basic properties of $\omega_f.$
\begin{proposition}
Let $T_1,\ldots,T_n\in\mathbb{B}(\mathscr{H})$ and let $f:[0,\infty)\to [0,\infty)$ be a continuous increasing function with $f(0)=0$. Then
\begin{itemize}
\item[(i)] ${{\omega }_{f}}\left( {{T}_{1}},\ldots ,{{T}_{n}} \right)=0\Leftrightarrow {{T}_{1}}=\cdots ={{T}_{n}}=0$.
\item[(ii)] ${{\omega }_{f}}\left( \alpha {{T}_{1}},\ldots ,\alpha {{T}_{n}} \right)=\left| \alpha \right|{{\omega }_{f}}\left( {{T}_{1}},\ldots ,{{T}_{n}} \right)$ for all $\alpha \in \mathbb{C}$, provided that $f$ is multiplicative.
\item[(iii)] ${{\omega }_{f}}\left( {{T}_{1}}+T_{1}^{'},\ldots ,{{T}_{n}}+T_{n}^{'} \right)\le {{\omega }_{f}}\left( {{T}_{1}},\ldots ,{{T}_{n}} \right)+{{\omega }_{f}}\left( T_{1}^{'},\ldots ,T_{n}^{'} \right)$, provided that $f$ is geometrically convex. That is, $f\left(\sqrt{ab}\right)\le \sqrt{f(a)f(b)}.$
\item[(iv)] ${{\omega }_{f}}\left( {{T}_{1}},\ldots ,{{T}_{n}} \right)={{\omega }_{f}}\left( {{T^*}_{1}},\ldots ,{{T^*}_{n}} \right).$
\item[(v)] If ${{U}_{1}},\ldots ,{{U}_{n}}$ are unitary, then
\[{{\omega }_{f}}\left( U_{1}^{*}{{T}_{1}}{{U}_{1}},\ldots ,U_{n}^{*}{{T}_{n}}{{U}_{n}} \right)={{\omega }_{f}}\left( {{T}_{1}},\ldots ,{{T}_{n}} \right).\]

\item[(vi)] If $g:[0,\infty)\to [0,\infty)$ is an injective function such that $g(0)=0$, and $f\circ g^{-1}$ is convex, then $$\omega_f(T_1,\ldots,T_n)\le \omega_g(T_1,\ldots,T_n).$$
\end{itemize}
\end{proposition}
\begin{proof}
The first, second, and fourth assertions immediately follow the definition of $\omega_f.$ For (iii), assume that $f$ is an increasing geometrically convex function. Then
\[\begin{aligned}
\sum\limits_{j=1}^{n}{f\left( \left| \left\langle \left( {T_j}+{T_{j}^{'}} \right)x,x \right\rangle \right| \right)}&=\sum\limits_{j=1}^{n}{f\left( \left| \left\langle {T_j}x,x \right\rangle +\left\langle {{T_{j}^{'}}}x,x \right\rangle \right| \right)} \\
& \le\sum_{j=1}^{n}{f\left( \left| \left\langle {T_j}x,x \right\rangle \right|+\left| \left\langle {{T_{j}^{'}}}x,x \right\rangle \right| \right)},
\end{aligned}\]
where we obtain the last inequality by the triangle inequality and the fact that $f$ is increasing. 
On the other hand, since $f$ is geometrically convex, it follows that  \cite[Corollary 1.1]{1}  
\[{{f}^{-1}}\left(\sum_{j=1}^{n}{f\left( \left| \left\langle {T_j}x,x \right\rangle \right|+\left| \left\langle {{T_{j}^{'}}}x,x \right\rangle \right| \right)} \right)\le {{f}^{-1}}\left(\sum_{j=1}^{n}{f\left( \left| \left\langle {T_j}x,x \right\rangle \right| \right)} \right)+{{f}^{-1}}\left(\sum_{j=1}^{n}{f\left( \left| \left\langle {{T_{j}^{'}}}x,x \right\rangle \right| \right)} \right),\]
which implies,
\[{{f}^{-1}}\left(\sum_{j=1}^{n}{f\left( \left| \left\langle \left( {T_j}+{{T_{i}^{'}}} \right)x,x \right\rangle \right| \right)} \right)\le {{f}^{-1}}\left(\sum_{j=1}^{n}{f\left( \left| \left\langle {T_j}x,x \right\rangle \right| \right)} \right)+{{f}^{-1}}\left(\sum_{j=1}^{n}{f\left( \left| \left\langle {{T_{j}^{'}}}x,x \right\rangle \right| \right)} \right).\]
Consequently,
\[{{\omega }_{f}}\left( {{T}_{1}}+{{T_{1}^{'}}},\ldots ,{{T}_{n}}+{{T_{n}^{'}}} \right)\le {{\omega }_{f}}\left( {{T}_{1}},\ldots ,{{T}_{n}} \right)+{{\omega }_{f}}\left( {{T_{1}^{'}}},\ldots ,{{T_{n}^{'}}} \right).\]
To prove (v), we have
\[\begin{aligned}
   {{\omega }_{f}}\left( U_{1}^{*}{{T}_{1}}{{U}_{1}},\ldots ,U_{n}^{*}{{T}_{n}}{{U}_{n}} \right)&=\underset{\left\| x \right\|=1}{\mathop{\sup }}\,{{f}^{-1}}\left( \sum\limits_{j=1}^{n}{f\left( \left| \left\langle U_{j}^{*}{{T}_{j}}{{U}_{j}}x,x \right\rangle  \right| \right)} \right) \\ 
 & =\underset{\left\| x \right\|=1}{\mathop{\sup }}\,{{f}^{-1}}\left( \sum\limits_{j=1}^{n}{f\left( \left| \left\langle {{T}_{j}}{{U}_{j}}x,{{U}_{j}}x \right\rangle  \right| \right)} \right) \\ 
 & =\underset{\left\| y \right\|=1}{\mathop{\sup }}\,{{f}^{-1}}\left( \sum\limits_{j=1}^{n}{f\left( \left| \left\langle {{T}_{j}}y,y \right\rangle  \right| \right)} \right) \\ 
 &= {{\omega }_{f}}\left( {{T}_{1}},\ldots ,{{T}_{n}} \right).  
\end{aligned}\]

Finally, for (vi), we note first that convexity of $f\circ{{g}^{-1}}$, together with the facts that $f(0)=g(0)=0$, implies
\[fo{{g}^{-1}}\left( a \right)+f\circ{{g}^{-1}}\left( b \right)\le f\circ{{g}^{-1}}\left( a+b \right); \; a,b\ge 0\]
thanks to Lemma \ref{superadditive}. Since ${{f}^{-1}}$ is an increasing function, then
	\[{{f}^{-1}}\left( f\circ{{g}^{-1}}\left( a \right)+f\circ{{g}^{-1}}\left( b \right) \right)\le {{g}^{-1}}\left( a+b \right).\]
Now, replacing $a$ and $b$ by $g\left( a \right)$ and $g\left( b \right)$, we get
	\[{{f}^{-1}}\left( f\left( a \right)+f\left( b \right) \right)\le {{g}^{-1}}\left( g\left( a \right)+g\left( b \right) \right).\]
The last inequality can be extended to $n$-tuple as follows
\[{{f}^{-1}}\left( \sum\limits_{j=1}^{n}{f\left( {{a}_{j}} \right)} \right)\le {{g}^{-1}}\left( \sum\limits_{j=1}^{n}{g\left( {{a}_{j}} \right)} \right); \;  a_j\ge 0.\]
Now, let $x\in\mathscr{H}$ be a unit vector. Replacing $a_j$ in the above inequality by $|\left<T_jx,x\right>|$, then taking the supremum implies 
\[{{\omega }_{f}}\left( {{T}_{1}},\ldots ,{{T}_{n}} \right)\le {{\omega }_{g}}\left( {{T}_{1}},\ldots ,{{T}_{n}} \right).\]
This completes the proof.
\end{proof}

Next, we attempt to find a relation between ${{\omega }_{f}}$ and $\omega $.

\begin{theorem}\label{6}
Let $T_1,\ldots,T_n\in \mathbb{B}(\mathscr{H})$ and let $f:[0,\infty)\to [0,\infty)$ be a continuous increasing convex  function with $f\left( 0 \right)=0$. Then
\begin{equation*}
{{\omega }_{f}}\left( {{T}_{1}},\ldots ,{{T}_{n}} \right)\le\sum_{j=1}^{n}{\omega \left( {T_j} \right)}.
\end{equation*}
\end{theorem}
\begin{proof}
Since $f$ is convex increasing, it follows that $f^{-1}$ is increasing and concave. By Lemma \ref{superadditive}, we have
	\begin{equation}\label{superadditive2}
f^{-1}(a+b)\le f^{-1}(a)+f^{-1}(b); \; a,b\ge 0.
	\end{equation}
Further, since $f$ is convex, superadditivity of $f$ implies 
\[\sum\limits_{j=1}^{n}{f\left( {{a}_{j}} \right)}\le f\left(\sum_{j=1}^{n}{{{a}_{j}}} \right)\]
for any ${{a}_{j}}\in J$. Monotony of $f^{-1}$ then implies 
\begin{equation}\label{1}
{{f}^{-1}}\left(\sum_{j=1}^{n}{f\left( {{a}_{j}} \right)} \right)\le\sum_{j=1}^{n}{{{a}_{j}}}.
\end{equation}
By replacing ${{a}_{i}}$ by $\left| \left\langle {T_j}x,x \right\rangle \right|$ in \eqref{1}, we obtain
\[{{f}^{-1}}\left(\sum_{j=1}^{n}{f\left( \left| \left\langle {T_j}x,x \right\rangle \right| \right)} \right)\le\sum_{j=1}^{n}{\left| \left\langle {T_j}x,x \right\rangle \right|},\]
for all unit vectors $x\in\mathscr{H}.$
Now, by taking supremum over unit vectors $x\in \mathscr{H}$, we get
\begin{equation}\label{s1}
{{\omega }_{f}}\left( {{T}_{1}},\ldots ,{{T}_{n}} \right)\le\sum_{j=1}^{n}{\omega \left( {T_j} \right)},
\end{equation}
as desired.
\end{proof}
\begin{remark}
For any $x\in \mathscr{H}$ with $\|x\|=1$, it holds
\begin{equation*}
|\left\langle {T_j}x,x \right\rangle|\le \omega \left( {T_j} \right).
\end{equation*}
If $f:[0,\infty)\to [0,\infty)$ is increasing, we get
\begin{equation*}
\sum\limits_{j=1}^{n}{f\left( \left| \left\langle {T_j}x,x \right\rangle \right| \right)}\le\sum_{j=1}^{n}{f\left( \omega \left( {T_j}\right)\right)}.
\end{equation*}
This implies
\begin{equation*}
	\omega_f\left( {{T}_{1}},\ldots ,{{T}_{n}} \right)=\sup_{\|x\|=1}f^{-1}\left( \sum_{j=1}^{n}f\left( \left| \left\langle {T_j}x,x \right\rangle \right| \right) \right) \le f^{-1}\left( \sum_{j=1}^{n}f\left( \omega \left( {T_j}\right)\right)\right).
	\end{equation*}
Now, if $f$ is convex (and increasing of course),  $f^{-1}$ is concave (and increasing), hence $f^{-1}$ is subadditive. That is
\begin{equation*}
f^{-1}\left( \sum_{j=1}^{n}f\left( \omega \left( {T_j}\right)\right)\right)\le \sum_{j=1}^{n} f^{-1}\left(f\left( \omega \left( {T_j}\right)\right)\right)=\sum\limits_{j=1}^{n}{\omega \left( {T_j} \right)}.
\end{equation*}
Thus, we have shown that if $T_j\in\mathbb{B}(\mathscr{H})$ and $f:[0,\infty)\to [0,\infty)$ is a continuous increasing convex function  then
\begin{equation}\label{eq_w_f_1}
\omega_f(T_1,\ldots,T_n)\le f^{-1}\left( \sum_{j=1}^{n}f\left( \omega \left( {T_j}\right)\right)\right)\le \sum_{j=1}^{n} f^{-1}\left(f\left( \omega \left( {T_j}\right)\right)\right)=\sum\limits_{j=1}^{n}{\omega \left( {T_j} \right)}.
\end{equation}
This indeed provides a considerable refinement of \eqref{s1}. We notice that the condition $f(0)$ is unnecessary here.
\end{remark}

In the following theorem, we present the $\omega_f$ version of the first inequality in \eqref{eq_eq_w_norm}. We notice that \eqref{s1} provides the $\omega_f$ version of the second inequality in \eqref{eq_eq_w_norm} because $\omega(T_j)\le \|T_j\|.$ In fact, \eqref{eq_w_f_1} provides further details than \eqref{s1}. However, we need to be cautious here as \eqref{eq_w_f_1} is valid for convex functions, while the next is for concave functions.

\begin{theorem}\label{thm_w_f_norm}
Let $T_1,\ldots,T_n\in \mathbb{B}(\mathscr{H})$ and let $f:[0,\infty)\to [0,\infty)$ be a continuous increasing concave function with  $f\left( 0 \right)=0$. Then
\begin{equation*}
\frac{1}{2}\left\|\sum_{j=1}^{n}{{T_j}} \right\|\le \omega\left(\sum_{j=1}^{n}T_j\right)\le {{\omega }_{f}}\left( {{T}_{1}},\ldots ,{{T}_{n}} \right).
\end{equation*}
\end{theorem}
\begin{proof}
Let $x\in\mathscr{H}$ be a unit vector. Since $f$ is concave with $f(0)=0$, $f^{-1}$ is convex with $f^{-1}(0)=0.$ Applying Lemma \ref{superadditive}, we have
\begin{align*}
\omega_f(T_1,\cdots,T_n)&\ge f^{-1}\left(\sum_{j=1}^{n}f(|\left<T_jx,x\right>|)\right)\\
&\ge \sum_{j=1}^{n}|\left<T_jx,x\right>|\\
&\ge \left| \sum_{j=1}^{n}\left<T_jx,x\right>\right|\\
&=\left|\left<\left(\sum_{j=1}^{n}T_j\right)x,x\right>\right|.
\end{align*}
Taking the supremum over unit vectors $x\in\mathscr{H}$, we obtain $\omega_f(T_1,\cdots,T_n)\ge \omega\left(\sum_{j=1}^{n}T_j\right)$. The result follows immediately from \eqref{eq_eq_w_norm}.
\end{proof}

The following result is concerned with some lower bounds for $\omega_{f}(\cdot).$

	\begin{proposition}
	Let $T_1,\ldots,T_n\in \mathbb{B}(\mathscr{H})$ and let $f:[0,\infty)\to [0,\infty)$ be an increasing convex   function. Then
	\begin{equation}\label{lowerbound2}
{{\omega }_{f}}\left( {{T}_{1}},\ldots ,{{T}_{n}}\right) \ge \sup_{|\lambda_j|\le 1}\omega\left( \sum_{j=1}^n \frac{\lambda_j}{n}T_j\right)\ge \frac{1}{2}\sup_{|\lambda_j|\le 1}\left\| \sum_{j=1}^n \frac{\lambda_j}{n}T_j\right\|.
	\end{equation}
	\end{proposition}
	\begin{proof}
	By  convexity of $f$ we have, for any $\lambda_j \in \mathbb{C}$ with $|\lambda_j|\le 1$ and any unit vector $x\in\mathscr{H},$
		\begin{align*}
	f^{-1}\left(\displaystyle\sum_{j=1}^nf\left(|\big\langle T_jx, x\big\rangle|\right)\right)
	&\ge \displaystyle\sum_{j=1}^n \frac{1}{n}|\big\langle T_jx, x\big\rangle|\nonumber\\
	&\ge \left|\sum_{j=1}^n \Big\langle \frac{\lambda_j}{n}T_jx, x\Big\rangle \right|\\
	&=\left| \Big\langle \sum_{j=1}^n \frac{\lambda_j}{n} T_jx, x\Big\rangle \right|.\nonumber \
	\end{align*}
	Taking the supremum over $x\in \mathscr{H}$ with $\|x\|=1$ yields 
	\begin{equation*}
{{\omega }_{f}}\left( {{T}_{1}},\ldots ,{{T}_{n}}\right)\ge \omega\left( \sum_{j=1}^n \frac{\lambda_j}{n} T_j\right),
	\end{equation*}
	for any $\lambda=(\lambda_1, \ldots, \lambda_n)$ with $|\lambda_j|\le 1$. Therefore,
	\begin{equation*}
	{{\omega }_{f}}\left( {{T}_{1}},\ldots ,{{T}_{n}}\right)\ge \sup_{|\lambda_j|\le 1}\omega\left( \sum_{j=1}^n \frac{\lambda_j}{n} T_j\right).
	\end{equation*}
 The second inequality follows quickly from \eqref{eq_eq_w_norm}.
	\end{proof}

On making use of inequality \eqref{lowerbound2}, we find different lower bounds for $ {{\omega }_{f}}$.

\begin{corollary}\label{3}
 Let $T_1,\ldots,T_n\in \mathbb{B}(\mathscr{H})$ and let $f:[0,\infty)\to [0,\infty)$ be an increasing convex function. Then
	\begin{equation*}
	{{\omega }_{f}}\left( {{T}_{1}},\ldots ,{{T}_{n}}\right) \ge \frac 1n\max\{\omega(T_1),\ldots, \omega(T_n)\}\ge \frac {1}{2n}\max\{\|T_1\|,\ldots, \|T_n\|\}.
	\end{equation*}
	\end{corollary}
\begin{proof} 
For any $j\in \{1, \ldots, n\}$, we consider $\lambda=(\lambda_1, \ldots, \lambda_n) \in \mathbb{C}^n$ such that $\lambda_i=1$ and $\lambda_j=0$ if $j\neq i$. Then, by \eqref{lowerbound2}, we have
	\[{{\omega }_{f}}\left( {{T}_{1}},\ldots ,{{T}_{n}}\right) \ge \frac 1n\omega(T_j)\ge \frac{1}{2n}\|T_j\|,\]
	for any $1\le j\le n$, and this completes the proof. 
\end{proof}

\begin{corollary}\label{4}
	Let $T_1,\ldots,T_n\in \mathbb{B}(\mathscr{H})$ and let $f:[0,\infty)\to[0,\infty)$ be an increasing convex  function. Then
	\begin{equation*}
	{{\omega }_{f}}\left( {{T}_{1}},\ldots ,{{T}_{n}}\right) \ge \frac 1n \max\left\{\omega\left( \sum_{j=1}^n \pm T_j\right)\right\}\ge \frac{1}{2n} \max\left\{\left\| \sum_{j=1}^n \pm T_j\right\|\right\}.
	\end{equation*}
\end{corollary}
\begin{proof}
	It is a simple consequence of \eqref{lowerbound2} where we consider $\lambda_j=\pm 1$ for  $1\le j\le n.$
\end{proof}

In the previous statement we can consider $\lambda_j=e^{\mathrm i\theta}$ with $\theta\in [0, 2\pi].$

\begin{remark}
From Corollary \ref{4}, we get
	\[{{\omega }_{f}}\left( {{T}_{1}},{{T}_{2}} \right)\ge \frac{1}{2}\omega \left( {{T}_{1}}+{{T}_{2}} \right).\]
Let $T=B+\mathrm iC$ be the Cartesian decomposition of the operator $T\in \mathbb B\left( \mathscr H \right)$. Setting ${{T}_{1}}=B$ and ${{T}_{2}}=\mathrm iC$, we infer that
	\[{{\omega }_{f}}\left( B,C \right)={{\omega }_{f}}\left( B,\mathrm iC \right)\ge \frac{1}{2}\omega \left( B+\mathrm iC \right)=\frac{1}{2}\omega \left( T \right).\]
\end{remark}

\begin{remark}
Letting ${{T}_{1}}={{T}_{2}}=\cdots ={{T}_{n}}=T$. From Theorem \ref{6}, we get
\begin{equation}\label{7}
{{\omega }_{f}}\left( T,\ldots ,T \right)\le n~\omega \left( T \right).
\end{equation}
On the other hand, by Corollary \ref{4}, we infer that
\begin{equation}\label{8}
{{\omega }_{f}}\left( T,\ldots ,T \right)\ge \omega \left( T \right).
\end{equation}
Combining two inequalities \eqref{7} and \eqref{8}, we reach to
	\[\omega \left( T \right)\le {{\omega }_{f}}\left( T,\ldots ,T \right)\le n~\omega \left( T \right).\]
\end{remark}

In the following, we present a lower bound for the generalized Davis-Wielandt radius introduced in the introduction.
\begin{corollary}\label{5}
Let $T \in \mathbb B(\mathscr H)$ and let $f:[0,\infty)\to [0,\infty)$ be a continuous increasing concave function with  $f\left( 0 \right)=0$. Then
\[\left\| \mathfrak RT+{{T}^{*}}T \right\|+\frac{\left| \omega \left( T+{{T}^{*}}T \right)-\omega \left( {{T}^{*}}+{{T}^{*}}T \right) \right|}{2}\le {{\omega }_{f}}\left( T,{{T}^{*}}T \right).\]
\end{corollary}
\begin{proof}
From Theorem \ref{thm_w_f_norm}, we have
\[\omega \left( T+{{T}^{*}}T \right)\le {{\omega }_{f}}\left( T,{{T}^{*}}T \right).\]
Since $\omega \left( X \right)=\omega \left( {{X}^{*}} \right)$ for any $X \in \mathbb B(\mathscr H)$, we get
\[\omega \left( {{T}^{*}}+{{T}^{*}}T \right)\le {{\omega }_{f}}\left( T,{{T}^{*}}T \right).\]
Thus,
\[\begin{aligned}
  & \left\| \mathfrak RT+{{T}^{*}}T \right\|+\frac{\left| \omega \left( T+{{T}^{*}}T \right)-\omega \left( {{T}^{*}}+{{T}^{*}}T \right) \right|}{2} \\ 
 & =\omega \left( \mathfrak RT+{{T}^{*}}T \right)+\frac{\left| \omega \left( T+{{T}^{*}}T \right)-\omega \left( {{T}^{*}}+{{T}^{*}}T \right) \right|}{2} \\ 
 & \le \frac{\omega \left( T+{{T}^{*}}T \right),\omega \left( {{T}^{*}}+{{T}^{*}}T \right)}{2}+\frac{\left| \omega \left( T+{{T}^{*}}T \right)-\omega \left( {{T}^{*}}+{{T}^{*}}T \right) \right|}{2} \\ 
 & =\max \left\{ \omega \left( T+{{T}^{*}}T \right),\omega \left( {{T}^{*}}+{{T}^{*}}T \right) \right\} \\ 
 & \le {{\omega }_{f}}\left( T,{{T}^{*}}T \right), 
\end{aligned}\]
as desired.
\end{proof}

We notice that Corollary \ref{5} provides some possible relation between $\omega_f(T,T^*T)$ and $\|\mathfrak{R}T+T^*T\|$ when $f$ is a concave function. In contrast, the following corollary presents a possible relation between these quantities when $f$ is convex.

\begin{corollary}
	Let $T\in \mathbb{B}\left( \mathscr{H} \right)$ and let $f:[0,\infty)\to[0,\infty)$ be an increasing convex  function. Then
	\[\frac{1}{2}\max \left\{ \omega \left( T \right),{{\left\| T \right\|}^{2}} \right\}\le {{\omega }_{f}}\left( T,{{T}^{*}}T \right),\]
and
\[\frac{1}{2}\left\| \mathfrak RT+{{T}^{*}}T \right\|+\frac{\left| \omega \left( T+{{T}^{*}}T \right)-\omega \left( {{T}^{*}}+{{T}^{*}}T \right) \right|}{4}\le {{\omega }_{f}}\left( T,{{T}^{*}}T \right).\]
\end{corollary}
\begin{proof}
Employing Corollary \ref{3}, gives
\[\begin{aligned}
   {{\omega }_{f}}\left( T,{{T}^{*}}T \right)&\ge \frac{1}{2}\max \left\{ \omega \left( T \right),\omega \left( {{T}^{*}}T \right) \right\} \\ 
 & =\frac{1}{2}\max \left\{ \omega \left( T \right),\left\| {{T}^{*}}T \right\| \right\} \\ 
 & =\frac{1}{2}\max \left\{ \omega \left( T \right),{{\left\| T \right\|}^{2}} \right\}.  
\end{aligned}\]
This proves the first inequality. To establish the second inequality, by Corollary \ref{4}, we have
\[{{\omega }_{f}}\left( T,{{T}^{*}}T \right)\ge \frac{1}{2}\omega \left( T+{{T}^{*}}T \right).\]
Applying the same arguments as in the proof of Corollary \ref{5} indicates the expected result.
\end{proof}

\section{More elaborated relations with the numerical radius }
In 1994, Furuta \cite{Furuta} proved an attractive generalization of Kato's (Cauchy--Schwarz) inequality, for an arbitrary $T\in\mathbb{B}(\mathscr{H})$, as follows
\begin{align}
	\left| {\left\langle {T\left| T \right|^{\alpha + \beta - 1} x,y} \right\rangle } \right|^2 \le \left\langle {\left| T \right|^{2\alpha } x,x} \right\rangle \left\langle {\left| T^* \right|^{2\beta } y,y} \right\rangle \label{eq3.3}
\end{align}
for any $x, y \in \mathscr{H}$ and $\alpha,\beta\in \left[0,1\right]$ with $\alpha+\beta \ge1$.

In the following result, we present an upper bound of $\omega_f$ for operators of the form $T|T|^{\alpha+\beta-1}$ appearing in \eqref{eq3.3}.
\begin{theorem}
Let $T_1,\ldots,T_n\in \mathbb{B}(\mathscr{H})$ and let $f:[0,\infty)\to [0,\infty)$ be an increasing continuous  geometrically convex function. If $p,q>1$ are such that $\frac{1}{p}+\frac{1}{q}=1$, then
\[{{\omega }_{f}}\left( {{T}_{1}}{{\left| {{T}_{1}} \right|}^{\alpha +\beta -1}},\ldots ,{{T}_{n}}{{\left| {{T}_{n}} \right|}^{\alpha +\beta -1}} \right)\le \left\| {{f}^{-1}}\left( \sum\limits_{j=1}^{n}{\left( \frac{1}{p}{{f}^{\frac{p}{2}}}\left( {{\left| {{T}_{j}} \right|}^{2\alpha }} \right)+\frac{1}{q}{{f}^{\frac{q}{2}}}\left( {{\left| T_{j}^{*} \right|}^{2\beta }} \right) \right)} \right) \right\|,\]
for any $\alpha,\beta\in \left[0,1\right]$ with $\alpha+\beta \ge1$.
\end{theorem}
\begin{proof}
Employing \eqref{eq3.3} for the $n$-tuple operators $\left(T_1,\ldots,T_n\right)$, by setting $y=x$, we have
\begin{align*}
\sum\limits_{i=1}^{n}{	f\left(\left| {\left\langle {T_j\left| T_j \right|^{\alpha + \beta - 1} x,x} \right\rangle } \right| \right)}&\le\sum_{j=1}^{n}{f\left(\left\langle {\left| T_j \right|^{2\alpha } x,x} \right\rangle^{\frac{1}{2}} \left\langle {\left| T^*_j \right|^{2\beta } x,x} \right\rangle^{\frac{1}{2}} \right)}
\\
&\le\sum_{j=1}^{n}{f^{\frac{1}{2}}\left(\left\langle {\left| T_j \right|^{2\alpha } x,x} \right\rangle\right) f^{\frac{1}{2}}\left(\left\langle {\left| T^*_j \right|^{2\beta } x,x} \right\rangle \right)}
\\
&\le \left(\sum\limits_{j=1}^{n}{f^{\frac{p}{2}}\left(\left\langle {\left| T_j \right|^{2\alpha } x,x} \right\rangle\right) }\right)^{\frac{1}{p}}
\left(\sum\limits_{j=1}^{n}{f^{\frac{q}{2}}\left(\left\langle {\left| T^*_j \right|^{2\beta } x,x} \right\rangle\right) }\right)^{\frac{1}{q}}
\\
&\le \frac{1}{p}\sum_{j=1}^{n}{f^{\frac{p}{2}}\left(\left\langle {\left| T_j \right|^{2\alpha } x,x} \right\rangle\right) }
+\frac{1}{q}\sum_{j=1}^{n}{f^{\frac{q}{2}}\left(\left\langle {\left| T^*_j \right|^{2\beta } x,x} \right\rangle\right) }.
\end{align*}
Thus,
{\small
\begin{align*}
f^{-1}\left(	\sum\limits_{j=1}^{n}{	f\left(\left| {\left\langle {T_j\left| T_j \right|^{\alpha + \beta - 1} x,x} \right\rangle } \right| \right)} \right)&\le f^{-1}\left( \frac{1}{p}\sum_{j=1}^{n}{f^{\frac{p}{2}}\left(\left\langle {\left| T_j \right|^{2\alpha } x,x} \right\rangle\right) }
+\frac{1}{q}\sum_{j=1}^{n}{f^{\frac{q}{2}}\left(\left\langle {\left| T^*_j \right|^{2\beta } x,x} \right\rangle\right) } \right).
\end{align*}
}
We get the required result by taking the supremum over all unit vector $x\in \mathscr{H}$.
\end{proof}

A more straightforward upper bound of $\omega_f$ can be stated as follows.

\begin{theorem}\label{thm_w_f_finverse}
	Let $T_1,\ldots,T_n\in \mathbb{B}(\mathscr{H})$ and let $f:[0,\infty)\to [0,\infty)$ be an increasing convex function. Then
	\[{{\omega }_{f}}\left( {{T}_{1}},\ldots ,{{T}_{n}} \right)\le \left\| {{f}^{-1}}\left( \sum\limits_{j=1}^{n}{\left( \frac{f\left( {{\left| {{T}_{j}} \right|}^{2\alpha }} \right)+f\left( {{\left| T_{j}^{*} \right|}^{2\left( 1-\alpha  \right)}} \right)}{2} \right)} \right) \right\|,\]
for any $0\le \alpha\le 1$.
\end{theorem}
\begin{proof}
For $0\le \alpha\le 1,$ the Cauchy-Schwarz inequality, together with the arithmetic-geometric mean inequality, implies
\begin{align*}
	\left| {\left\langle {T_j x,x} \right\rangle } \right| &\le \left\langle {\left| {T_j } \right|^{2\alpha } x,x} \right\rangle^{\frac{1}{2}} \left\langle {\left| {T_j^* } \right|^{2\left( {1 - \alpha } \right)} x,x} \right\rangle^{\frac{1}{2}}
	\\
	&\le
	\left\langle {\frac{{\left| {T_j } \right|^{2\alpha } + \left| {T_j^* } \right|^{2\left( {1 - \alpha } \right)} }}{2}x,x} \right\rangle,
\end{align*}
for the unit vector $x\in\mathscr{H}.$ 

Noting that $f$ is increasing, then applying Lemma \ref {lem_conv_inner} we have
\begin{align*}
\sum\limits_{j = 1}^n {f\left( {\left| {\left\langle {T_j x,x} \right\rangle } \right|} \right)}
&\le \sum\limits_{j = 1}^n {f\left(\left\langle {\left( { \frac{\left| {T_j } \right|^{2\alpha} + \left| {T_j^* } \right|^{2\left( {1 - \alpha } \right)}}{2}} \right)x,x} \right\rangle\right) }
\\
	&\le \sum\limits_{j = 1}^n { \left\langle {f\left( { \frac{\left| {T_j } \right|^{2\alpha} + \left| {T_j^* } \right|^{2\left( {1 - \alpha } \right)}}{2}} \right)x,x} \right\rangle }
	\\
	&\le
	\sum\limits_{j = 1}^n {\left\langle {\left( {\frac{{f\left( {\left| {T_j } \right|^{2\alpha } } \right) + f\left( {\left| {T_j^* } \right|^{2\left( {1 - \alpha } \right)} } \right)}}{2}} \right)x,x} \right\rangle }
	\\
	&=
	\left\langle {\sum\limits_{j = 1}^n {\left( {\frac{{f\left( {\left| {T_j } \right|^{2\alpha } } \right) + f\left( {\left| {T_j^* } \right|^{2\left( {1 - \alpha } \right)} } \right)}}{2}} \right)x} ,x} \right\rangle,
\end{align*}
which implies 
\begin{align*}
f^{-1}\left(	\sum\limits_{j = 1}^n {f\left( {\left| {\left\langle {T_j x,x} \right\rangle } \right|} \right)} \right)
&\le f^{-1}\left( \left\langle {\sum\limits_{j = 1}^n {\left( {\frac{{f\left( {\left| {T_j } \right|^{2\alpha } } \right) + f\left( {\left| {T_j^* } \right|^{2\left( {1 - \alpha } \right)} } \right)}}{2}} \right)x} ,x} \right\rangle \right).
\end{align*}
We get the required result by taking the supremum over all unit vectors $x\in \mathscr{H}$, noting that $f^{-1}$ is also increasing.
\end{proof}

Another bound, similar to that in Theorem \ref{thm_w_f_finverse}, can be stated as follows. The proof is very similar to that of Theorem \ref{thm_w_f_finverse}, so we do not include it here.
\begin{theorem}\label{thm_w_f_finverse}
	Let $T_1,\ldots,T_n\in \mathbb{B}(\mathscr{H})$ and let $f:[0,\infty)\to [0,\infty)$ be an increasing convex function. If $p_j>0$ so that $\sum_{j=1}^{n}p_j=1$, then
	\begin{align*}
	{{\omega }_{f}}\left(p_1 {{T}_{1}},\ldots ,p_n{{T}_{n}} \right) \le  \left\|f^{-1}\left( \sum\limits_{j = 1}^n p_j\left( {\frac{{f\left( {\left| {T_j } \right|^{2\alpha } } \right) + f\left( {\left| {T_j^* } \right|^{2\left( {1 - \alpha } \right)} } \right)}}{2}} \right)\right) \right\|,	 
	\end{align*}
for any $0\le \alpha\le 1$.
\end{theorem}

In the following result, a super-multiplicative function refers to a function $f:[0,\infty)\to [0,\infty)$ such that $f(a)f(b)\le f(ab)$ for all $a,b\in [0,\infty).$ We notice that all power functions $f(t)=t^r, r>0$ are such functions.
\begin{theorem}\label{thm_w_f_10}
	Let $T_1,\ldots,T_n\in \mathbb{B}(\mathscr{H})$ and let $f:[0,\infty)\to [0,\infty)$ be an increasing, convex and super-multiplicative function. Then
\[{{\omega }_{f}}\left( {{T}_{1}},\ldots ,{{T}_{n}} \right)\le \left\| {{f}^{-1}}\left( \sqrt{n\sum\limits_{j=1}^{n}{f\left( \frac{T_{j}^{*}{{T}_{j}}+{{T}_{j}}T_{j}^{*}}{2} \right)}} \right) \right\|.\]	
\end{theorem}
\begin{proof}
Let $B_j+ \mathrm i C_j$ be the Cartesian decomposition of the Hilbert space operators $T_j$, for  $j=1,\cdots,n$. We have
\begin{align*}
	\left| {\left\langle {T_j x,x} \right\rangle } \right|^2 &= \left\langle {B_j x,x} \right\rangle ^2 + \left\langle {C_j x,x} \right\rangle ^2
	\\
	&\le \left\langle {B_j^2 x,x} \right\rangle + \left\langle {C_j^2 x,x} \right\rangle = \left\langle {\left( {B_j^2 + C_j^2 } \right)x,x} \right\rangle,
	\end{align*}
	where we have used Lemma \ref{lem_conv_inner} to obtain the last inequality, noting that both $B_j$ and $C_j$ are self-adjoint and that $f(t)=t^2$ is convex.
But since $f$ is increasing, super-multiplicative and convex, we have
	\begin{align*}
f^2\left( {\left| {\left\langle {T_j x,x} \right\rangle } \right| } \right) \le	f\left( {\left| {\left\langle {T_j x,x} \right\rangle } \right|^2 } \right) \le f\left( {\left\langle {\left( {B_j^2 + C_j^2 } \right)x,x} \right\rangle } \right) &\le \left\langle {f\left( {B_j^2 + C_j^2 } \right)x,x} \right\rangle
\end{align*}
which implies that
	\begin{align*}
	\sum\limits_{j = 1}^n {\left(f\left( {\left| {\left\langle {T_j x,x} \right\rangle } \right| } \right) \right)^2} \le \sum\limits_{j = 1}^n {f\left( {\left\langle {\left( {B_j^2 + C_j^2 } \right)x,x} \right\rangle } \right)} \le \sum\limits_{j = 1}^n {\left\langle {f\left( {B_j^2 + C_j^2 } \right)x,x} \right\rangle }.
\end{align*}
Applying Jensen's inequality to the function $g(t)=t^2$ implies 
	\begin{align*}
\frac{1}{n^2}\left(\sum\limits_{j = 1}^n {f\left( {\left| {\left\langle {T_j x,x} \right\rangle } \right| } \right) }\right)^2 &\le	\frac{1}{n}\sum\limits_{j = 1}^n {\left(f\left( {\left| {\left\langle {T_j x,x} \right\rangle } \right| } \right) \right)^2}\\
&\le \frac{1}{n}\sum\limits_{j = 1}^n {\left\langle {f\left( {B_j^2 + C_j^2 } \right)x,x} \right\rangle },
\end{align*}
and this is equivalent to 
	\begin{align*}
	\sum\limits_{j = 1}^n {f\left( {\left| {\left\langle {T_j x,x} \right\rangle } \right| } \right) } \le \left( n\sum\limits_{j = 1}^n {\left\langle {f\left( {B_j^2 + C_j^2 } \right)x,x} \right\rangle }\right) ^{\frac{1}{2}}.
\end{align*}
Also, since $f$ is increasing, we get
	\begin{align*}
	f^{ - 1} \left( {\sum\limits_{j = 1}^n {f\left( {\left| {\left\langle {T_j x,x} \right\rangle } \right|^2 } \right)} } \right) &\le f^{ - 1} \left( {\left( n\sum\limits_{j = 1}^n {\left\langle {f\left( {B_j^2 + C_j^2 } \right)x,x} \right\rangle }\right) ^{\frac{1}{2}}} \right)
\\
	&=
	f^{ - 1} \left( {\sqrt{n}\left\langle {\sum\limits_{j = 1}^n {f\left( {B_j^2 + C_j^2 } \right)x} ,x} \right\rangle^{\frac{1}{2}} } \right)
\\
	&=
	f^{ - 1} \left( {\sqrt{n}\left\langle {\sum\limits_{j = 1}^n {f\left( {\frac{T_j^*T_j + T_jT_j ^*}{2} } \right)x} ,x} \right\rangle^{\frac{1}{2}} } \right).
\end{align*}
We get the required result by taking the supremum over all unit vector $x\in \mathscr{H}$.
\end{proof}

In the following remark, we explain the significance of Theorem \ref{thm_w_f_10}.

\begin{remark}
Taking $f\left(t\right)=t^2$, $t\ge0$, Theorem \ref{thm_w_f_10} implies
	\begin{align}
	 {{\omega }_{\rm{e} }}\left( {{T}_{1}},\ldots ,{{T}_{n}} \right)\le
\sqrt{\frac{\sqrt{n}}{2}\left\| {\sum\limits_{j = 1}^n { \left( {T_j^*T_j + T_jT_j ^*} \right)^2} } \right\|^{\frac{1}{2}}}.\label{eq3.7}
\end{align}
In particular, choosing $n=1$ and $T_1=T$, we get
	\begin{align*}
	 {{\omega } }\left( T \right)\le
	\sqrt{\frac{1}{2}\left\| { T^*T + TT ^* } \right\| },
\end{align*}
or 
	\begin{align*}
	 {{\omega } }^2\left( T \right)\le
\frac{1}{2}\left\| { T^*T + TT ^* } \right\|,
\end{align*}
which is an outstanding result of Kittaneh \eqref{ineq_kitt_2}. A more general form of the inequality \eqref{eq3.7} could be stated by taking $f\left(t\right)=t^p$, $t\ge0$ $\left(p\ge1\right)$, in Theorem \ref{thm_w_f_10}
	\begin{align*}
{{\omega }_p^p}\left( {{T}_{1}},\ldots ,{{T}_{n}} \right)\le
\frac{\sqrt{n}}{2^{\frac{p}{2}}}\left\| {\sum\limits_{j = 1}^n { \left( {T_j^*T_j + T_jT_j ^*} \right)^p} } \right\|^{\frac{1}{2}}
\end{align*}
holds for all $p\ge1$.
\end{remark}

\begin{theorem}
Let $B_j+ \mathrm i C_j$ be the Cartesian decomposition of the Hilbert space operators $T_j\in \mathbb{B}\left(
\mathscr{H}\right)$ $(j=1,\ldots,n)$. Let $f:[0,\infty)\to [0,\infty)$ be an increasing convex function that satisfies  $f\left(0\right)=0$. Then
\begin{align*}
	{{\omega }_{f}}\left( {{T}_{1}},\ldots ,{{T}_{n}} \right)\le
 \left\| {{f}^{-1}}\left(\sum_{j=1}^{n}{f\left( \left| {{B}_{j}} \right|+\left| {{C}_{j}} \right| \right)} \right) \right\|.
\end{align*}
\end{theorem}
\begin{proof}
Let $B_j+ \mathrm i C_j$ be the Cartesian decomposition of the Hilbert space operators $T_j$ for all $j=1,\ldots,n$. If $x\in\mathscr{H}$ is a unit vector, we have
\[\begin{aligned}
   \sum\limits_{j=1}^{n}{f\left( \left| \left\langle {{T}_{j}}x,x \right\rangle  \right| \right)}&=\sum\limits_{j=1}^{n}{f\left( \sqrt{{{\left\langle {{B}_{j}}x,x \right\rangle }^{2}}+{{\left\langle {{C}_{j}}x,x \right\rangle }^{2}}} \right)} \\ 
 & \le \sum\limits_{j=1}^{n}{f\left( \left| \left\langle {{B}_{j}}x,x \right\rangle  \right|+\left| \left\langle {{C}_{j}}x,x \right\rangle  \right| \right)} \\ 
 & \le \sum\limits_{j=1}^{n}{f\left( \left\langle \left( \left| {{B}_{j}} \right|+\left| {{C}_{j}} \right| \right)x,x \right\rangle  \right)} \\ 
 & \le \sum\limits_{j=1}^{n}{\left\langle f\left( \left| {{B}_{j}} \right|+\left| {{C}_{j}} \right| \right)x,x \right\rangle }  
\end{aligned}\]
where we have used Lemma \ref{lem_conv_inner} twice to obtain the last two inequalities.
Thus, since $f$ is increasing,
	\[\begin{aligned}
   {{f}^{-1}}\left( \sum\limits_{j=1}^{n}{f\left( \left| \left\langle {{T}_{j}}x,x \right\rangle  \right| \right)} \right)&\le {{f}^{-1}}\left( \sum\limits_{j=1}^{n}{\left\langle f\left( \left| {{B}_{j}} \right|+\left| {{C}_{j}} \right| \right)x,x \right\rangle } \right) \\ 
 & ={{f}^{-1}}\left( \left\langle \left(\sum_{j=1}^{n}{f\left( \left| {{B}_{j}} \right|+\left| {{C}_{j}} \right| \right)} \right)x,x \right\rangle  \right) \\ 
 & \le {{f}^{-1}}\left( \left\|\sum_{j=1}^{n}{f\left( \left| {{B}_{j}} \right|+\left| {{C}_{j}} \right| \right)} \right\| \right) \\ 
 & = \left\| {{f}^{-1}}\left(\sum_{j=1}^{n}{f\left( \left| {{B}_{j}} \right|+\left| {{C}_{j}} \right| \right)} \right) \right\|,
\end{aligned}\]
where we obtain the last equality because $f$ is increasing.
\end{proof}

Now, extending \eqref{ineq_yama_1} to $\omega_f$, we have the following.

\begin{theorem}
Let $T_1,\ldots,T_n\in \mathbb{B}(\mathscr{H})$ and let $f:[0,\infty)\to [0,\infty)$ be an increasing convex function. Then
	\[{{\omega }_{f}}\left( {{T}_{1}},\ldots ,{{T}_{n}} \right)\le {{f}^{-1}}\left(\sum_{j=1}^{n}{\left( \frac{f\left( \left\| {T_j} \right\| \right)+f\left( \omega \left( \widetilde{{T_j}} \right) \right)}{2} \right)} \right).\]
\end{theorem}
\begin{proof}
For each $T_j$, let $T_j=U_j|T_j|$ be the polar decomposition of $T_j$. By Lemma \ref{lemma_theta}, if $x\in\mathscr{H}$ is a unit vector, it follows that  $|\left<T_jx,x\right>|\le \mathfrak{R}\left\{e^{\mathrm i\theta}\left<T_jx,x\right>\right\}$, for all $\theta\in\mathbb{R}$. Then, for all $\theta$, we have
\[\begin{aligned}
  & \left| \left\langle {T_j}x,x \right\rangle  \right| \\ 
 &\le\mathfrak R\left\{ {{e}^{\mathrm i{\theta}}}\left\langle {T_j}x,x \right\rangle  \right\} \\ 
 & =\frac{1}{4}\left\langle \left( {{e}^{-\mathrm i{\theta}}}+{U_j} \right)\left| {T_j} \right|\left( {{e}^{\mathrm i{\theta}}}+U_j^* \right)x,x \right\rangle -\frac{1}{4}\left\langle \left( {{e}^{-\mathrm i{\theta}}}-{U_j} \right)\left| {T_j} \right|\left( {{e}^{\mathrm i{\theta}}}-U_j^* \right)x,x \right\rangle  \\ 
 & \le \frac{1}{4}\left\langle \left( {{e}^{-\mathrm i{\theta}}}+{U_j} \right)\left| {T_j} \right|\left( {{e}^{\mathrm i{\theta}}}+U_j^* \right)x,x \right\rangle. 
\end{aligned}\]
Thus,
\[\begin{aligned}
  \sum_{j=1}^{n}{f\left( \left| \left\langle {T_j}x,x \right\rangle  \right| \right)}&\le\sum_{j=1}^{n}{f\left( \frac{1}{4}\left\langle \left( {{e}^{-\mathrm i{\theta}}}+{U_j} \right)\left| {T_j} \right|\left( {{e}^{\mathrm i{\theta}}}+U_j^* \right)x,x \right\rangle  \right)} \\ 
 & \le\sum_{j=1}^{n}{f\left( \left\langle \left( \frac{\left( {{e}^{-\mathrm i{\theta}}}+{U_j} \right)\left| {T_j} \right|\left( {{e}^{\mathrm i{\theta}}}+U_j^* \right)}{4} \right)x,x \right\rangle  \right)} \\ 
 & \le\sum_{j=1}^{n}{\left\langle f\left( \frac{\left( {{e}^{-\mathrm i{\theta}}}+{U_j} \right)\left| {T_j} \right|\left( {{e}^{\mathrm i{\theta}}}+U_j^* \right)}{4} \right)x,x \right\rangle } \\ 
 & \le\sum_{j=1}^{n}{f\left( \left\| \frac{\left( {{e}^{-\mathrm i{\theta}}}+{U_j} \right)\left| {T_j} \right|\left( {{e}^{\mathrm i{\theta}}}+U_j^* \right)}{4} \right\| \right)} \\ 
 & =\sum\limits_{i=1}^{n}{f\left( \left\| \frac{{{\left| {T_j} \right|}^{\frac{1}{2}}}\left( {{e}^{\mathrm i{\theta}}}+U_j^* \right)\left( {{e}^{-\mathrm i{\theta}}}+{U_j} \right){{\left| {T_j} \right|}^{\frac{1}{2}}}}{4} \right\| \right)} \\ 
 & =\sum\limits_{i=1}^{n}{f\left( \left\| \frac{2\left| {T_j} \right|+{{e}^{\mathrm i{\theta}}}\widetilde{{T_j}}+{{e}^{-\mathrm i{\theta}}}{{\left( \widetilde{{T_j}} \right)}^{*}}}{4} \right\| \right)} \\ 
 & =\sum\limits_{i=1}^{n}{f\left( \left\| \frac{\left| {T_j} \right|+\mathfrak R{{e}^{\mathrm i{\theta}}}\widetilde{{T_j}}}{2} \right\| \right)}.  
\end{aligned}\]
On the other hand,
\[\begin{aligned}
  \sum_{j=1}^{n}{\left\| f\left( \frac{\left| {T_j} \right|+{{\operatorname{\mathfrak Re}}^{\mathrm i{\theta}}}\widetilde{{T_j}}}{2} \right) \right\|}&\le \frac{1}{2}\sum\limits_{i=1}^{n}{\left\| f\left( \left| {T_j} \right| \right)+f\left( {{\operatorname{\mathfrak Re}}^{\mathrm i{\theta}}}\widetilde{{T_j}} \right) \right\|} \\ 
 & \le \frac{1}{2}\sum\limits_{i=1}^{n}{\left( \left\| f\left( \left| {T_j} \right| \right) \right\|+\left\| f\left( {{\operatorname{\mathfrak Re}}^{\mathrm i{{\theta }}}}\widetilde{{T_j}} \right) \right\| \right)} \\ 
 & =\frac{1}{2}\sum\limits_{i=1}^{n}{\left( f\left\| ~\left| {T_j} \right|~ \right\|+f\left( \left\| {{\operatorname{\mathfrak Re}}^{\mathrm i{\theta}}}\widetilde{{T_j}} \right\| \right) \right)} \\ 
 & \le \frac{1}{2}\sum\limits_{i=1}^{n}{\left( f\left( \left\| {T_j} \right\| \right)+f\left( \omega \left( \widetilde{{T_j}} \right) \right) \right).}  
\end{aligned}\]
So,
\[{{f}^{-1}}\left(\sum_{j=1}^{n}{f\left( \left| \left\langle {T_j}x,x \right\rangle  \right| \right)} \right)\le {{f}^{-1}}\left(\sum_{j=1}^{n}{\left( \frac{f\left( \left\| {T_j} \right\| \right)+f\left( \omega \left( \widetilde{{T_j}} \right) \right)}{2} \right)} \right),\]
which completes the proof.
\end{proof}

We close this paper by introducing an upper bound for the generalized Davis-Wielandt radius.
\begin{corollary}
Let $T\in \mathbb B\left( \mathscr H \right)$ with the polar decomposition $T=U|T|$ and let $f:[0,\infty)\to [0,\infty)$ be an increasing convex function. Then
\[{{\omega }_{f}}\left( T,{{T}^{*}}T \right)\le {{f}^{-1}}\left( \frac{f\left( \left\| T \right\| \right)+f\left( \omega \left( \widetilde{T} \right) \right)+f\left( {{\left\| T \right\|}^{2}} \right)+f\left( \omega \left( \left| T \right|U\left| T \right| \right) \right)}{2} \right).\]
\end{corollary}

\subsection*{Declarations}
\begin{itemize}
\item {\bf{Availability of data and materials}}: Not applicable
\item {\bf{Competing interests}}: The authors declare that they have no competing interests.
\item {\bf{Funding}}: Not applicable
\item {\bf{Authors' contributions}}: Authors declare that they have contributed equally to this paper. All authors have read and approved this version.
\end{itemize}




\vskip 0.5 true cm

\noindent{\tiny (M. W. Alomari) Department of Mathematics, Faculty of Science and Information Technology, Irbid National University,
Irbid 21110, Jordan}
	
\noindent	{\tiny\textit{E-mail address:} mwomath@gmail.com}

\vskip 0.3 true cm

\noindent{\tiny (M. Sababheh) Vice President, Princess Sumaya University for Technology, Amman, Jordan}
	
\noindent	{\tiny\textit{E-mail address:} sababheh@psut.edu.jo}

\vskip 0.3 true cm

\noindent{\tiny (C. Conde)  Instituto de Ciencias, Universidad Nacional de General Sarmiento  and  Consejo Nacional de Investigaciones Cient\'ificas y Tecnicas, Argentina}

\noindent{\tiny \textit{E-mail address:} cconde@campus.ungs.edu.ar}

\vskip 0.3 true cm

\noindent{\tiny (H. R. Moradi) Department of Mathematics, Payame Noor University (PNU), P.O. Box, 19395-4697, Tehran, Iran
	
\noindent	\textit{E-mail address:} hrmoradi@mshdiau.ac.ir}

\end{document}